\documentclass[12pt]{article}
\usepackage{amsmath,amssymb,amsthm,amscd,amsfonts}
\usepackage{xypic}
\newtheorem{thm}{Theorem}[section]
\newtheorem{defn}[thm]{Definition}
\newtheorem{prop}[thm]{Proposition}
\newtheorem{rmk}[thm]{Remark}
\newtheorem{cor}[thm]{Corollary}
\newtheorem{lem}[thm]{Lemma}
\newtheorem{asum}[thm]{Assumption}
\newcommand{\pd}{\mathrm{pardeg}}
\newcommand{\pas}{\mathrm{par}\mu}
\newcommand{\rk}{\mathrm{rank}}
\newcommand{\h}{\mathfrak{Hom}}
\newcommand{\e}{\mathrm{Ext}}
\newcommand{\Ph}{\mathrm{ParHom}}
\newcommand{\ho}{\mathrm{Hom}}
\newcommand{\Sph}{\mathrm{SParHom}}
\newcommand{\ph}{\mathfrak{ParHom}}
\newcommand{\sph}{\mathfrak{SParHom}}
\newcommand{\pc}{\mathrm{PChain}}
\newcommand{\pb}{\mathrm{PBundle}}
\title{On the motives of moduli of parabolic chains and parabolic Higgs bundles}
\author{Viet Cuong Do}
\begin{document}
\maketitle
\begin{abstract}
Like the Higgs bundles on a Riemann surface who played an important role in the study of representation of the fundamental group of the surface, the parabolic Higgs bundles play also their importance in the study of the fundamental group  but of the punctured surface. In this paper, we shall calculate the (virtual) motive (i.e in a suitable Grothendieck group) of the moduli spaces of parabolic bundles of fixed rank and fixed parabolic structure, using localization with respect to the circle action.  	
\end{abstract}

Let $C$ be a smooth curve of genius $g$ with $k$ marked points in the reduced divisor $D=p_1+\dots+ p_k$ and $E$ be a (holomorphic) bundle over $C$. Let $\Omega_C$ be the canonical bundle of $C$. A parabolic bundle $E$ is a bundle over $C$ equipped with weighted flags in the fibres of $E$ over the marked points. We can define for these parabolic bundles a notion of (semi-)stability which keeps track the weights. 

A parabolic Higgs bundle is a pair $(E,\Phi)$, where $E$ is a parabolic bundle and $\Phi: E\to E\otimes \Omega_C(D)$ is a strongly parabolic homomorphism, i.e the residue of $\Phi$ at $p_i\in D$ is strict triangular  with respect to the flag at this point. Naturally, the (semi-)stability above provides  a notion of (semi-)stability of the parabolic Higgs bundles. This notion allows  the construction of moduli spaces of semistable parabolic Higgs bundles of fixed rank , fixed degree and fixed weights.  For generic weights, semi-stability and stability coincide and the moduli space is a smooth quasi-projective algebraic manifold.

To calculate the motives of this moduli space we shall use the strategy of \cite{ghs} and the ``wall-cross" processing of \cite{gh} for bypassing the convergence problem which appear in \cite{ghs}. Let us now briefly review the structure of the paper and explain more precisely our strategy.

In the section 1, we recall the definition of the ring which we shall use to calculate our subjects. We end the section by recall the relation between the mixed Hodge polynomial and our formulas.

In the section 2, we collect some known results on parabolic Higgs bundles and in particular explain how the class of the moduli space of stable parabolic Higgs bundles is calculated from the classes of moduli spaces of parabolic chains - it's a point depart of our strategy. Here a parabolic chain is simply a collection $(E_0,\dots E_r)$ of parabolic vector bundles together with strongly parabolic maps $\phi_i:E_i\to E_{i-1}(D)$.

In the section 3, we do some study on the parabolic chains. The necessary conditions for the existence of semi-stable parabolic chains are in the section 3.2. The recursive formulas for Harder-Narasimhan stratum are given in the section 3.4. We give a calculation for the class of parabolic chains in a special cases in the section 3.5

The ``wall-crossing" process is explained in the section 4. Here we explain how the classical approach to study geometry of moduli spaces by variant the semi-stability in the point of view of stacks - which is easier in some sense. In particularly, we have an inductive expression for the difference between the moduli spaces of two different semi-stability. More precisely, this difference is the union of finite HN-stratum of smaller rank which is calculated in the section 3. 

The section 5 is glueing all together. Here we describe our algorithms to calculate the moduli spaces of stable parabolic chains. To make use of the section 4, we must introduce an extremal semi-stability, for which we know explicitly the moduli spaces of stable parabolic chains. Now is the time for playing the role of the section 3.2 and the calculation in special cases. 

Finally, the author believe that by varying the parabolic structure, we can extend the result of \cite{ghs} which is only for the rank and the degree are co-prime to the cases where they aren't co-prime any more from our result.

\textbf{Acknowledgment: }The author would like to thank J. Heinloth for introducing this problem and sharing his insights.


\section{Preliminaries}
Like the case non-parabolic, we shall do our computation in the ring $\hat{K_0}(\mathrm{Var})$. For the convenience, we shall recall its definition. 

We denote $K_0(\mathrm{Var}_k)$ the Grothendieck ring of varieties over $k$. We denote also $\mathbb{L}:=[\mathbb{A}^1]\in K_0(\mathrm{Var}_k)$. In $K_0(\mathrm{Var}_k)[\mathbb{L}^{-1}]$ we have the filtration defined by the subgroups generated by classes $[X]\mathbb{L}^{-m}$ with $\dim(X)-m\leq -n$ for  $n\in \mathbb{N}$ fixed. The completion of $K_0(\mathrm{Var}_k)$ according to this filtration is $\hat{K_0}(\mathrm{Var})$.

We note that $[GL_n]$ is an invertible element in $\hat{K_0}(\mathrm{Var})$.

To see the relation between our formulas of classes in $\hat{K_0}(\mathrm{Var})$ and the mixed Hodge polynomials we need to recall that (see \cite[Section 1.2]{ghs} for more detail)
\begin{prop}
We can easily read off the mixed Hodge polynomial from the classes in $\hat{K_0}(\mathrm{Var})$ if they can be expressed in terms of $\mathbb{L}$ and the symmetric power $C^{(i)}$ of the curve $C$. 
\end{prop}

\section{Parabolic Higgs bundles}
In this section we will collect the basic definitions on moduli space of parabolic Higgs bundles.

Let $C$ be a smooth curve of genius $g$ with $k$ marked points in the reduced divisor $D=p_1+\dots+ p_k$ and $E$ be a (holomorphic) bundle over $C$.
\begin{defn}\label{def1}\begin{itemize}
\item A \underline{parabolic} structure on $E$ consists of weighted flags:
$$\begin{matrix}
E_{p}&=&E_{p,1}&\supset&\dots&\supset&E_{p,s_p}&\supset&E_{p,(s_p+1)}=0\\
0&\leq&w_{p,1}&<&\dots&<&w_{p,m_p}&<&w_{p,s_p+1}=1
\end{matrix}$$
over each point $p\in D$.
\item A (holomorphic) map $\phi:E^1\to E^2$ between two parabolic bundles is called
\begin{itemize}
\item \underline{parabolic} if $w_{p,i}^1>w^2_{p,i'}$ implies that $\phi(E^1_{p,i})\subset E^2_{p,i'+1}\,\forall p\in D$,
\item
\underline{strongly parabolic} if $w_{p,i}^1\geq w^2_{p,i'}$  implies that $\phi(E^1_{p,i})\subset E^2_{p,i'+1}\,\forall p\in D$.
\end{itemize}
\end{itemize}
\end{defn}
If $\underline{s}=(s_{p_1},\dots, s_{p_k})$ and $\underline{w}=(w_{p,i}) $ for $p \in D$ and $1\leq i\leq s_p$, we say that the parabolic bundle $E$ has \textit{weight type} $(\underline{s},\underline{w})$. If $m_{p,i}=\dim(E_{p,i})-\dim(E_{p,i+1})$, for  $p\in D$ and $1\leq i\leq s_p$, we say that $E$ has \textit{weight mutiplicity }$\underline{m}=(m_{p,i}).$ We have $\sum_{i=1}^{s_p}m_{p,i}=\rk(E)\,\forall p\in D$.

An ordinary vector bundle $E$ is a parabolic bundle of weight type \linebreak $((1,\dots,1),(0,\dots,0))$ and of dimension vector $(\rk(E),\dots,\rk(E))$.

If $F$ is a subbundle of $E$ then $F$ inherits a parabolic structure from $E$ by pullback. More precisely, a parabolic structure on $F$ can be given by intersecting the flags with the fibres $F_{p}$, and discarding any subspace $E_{p,i}\cap F_{p}$ which coincides with $E_{p,i+1}\cap F_{p}$ and weight are assigned accordingly. Similarly, the quotient $E/F$ can be given a parabolic structure by projecting the flags to $E_{p}/F_{p}$. The weights of $E/F$ are precisely those discarded for $F$.
\begin{rmk}\label{rmkw}
For our inductive calculations, sometimes we need that a sub-quotient and a sub-bundle of a fixed parabolic bundle inherit the weight type of this fixed parabolic bundle, so we shall not discard the coincident subspaces of the flags. This convention doesn't give any confuse with the definition above (see. \cite[Remark 2.11]{yh}).     
\end{rmk}

We denote the set of homomorphisms (resp. parabolic homomorphisms and strongly parabolic homomorphisms) from $E$ to $F$ by $\ho(E,F)$ (resp. $\Ph(E,F)$ and $\Sph(E,F)$). We denote by $\ph(E,E')$ and $\sph(E,E')$ the sub-sheaves of $\h(E,E')$ formed by the parabolic and strongly parabolic homomorphisms from $E$ to $E'$, respectively.

For two parabolic bundles $E,F$, there is a well-defined notion of tensor product $E\otimes^P F$  \cite{y1}, which is best understood in terms of $\mathbb{R}$-filtered sheaves.  If one of the two parabolic bundles is an ordinary vector bundle- for example $F$, we have that $E\otimes^P F$ is the parabolic bundle $E\otimes F$. This parabolic bundle inherit  a weight type and a dimension vector from parabolic bundle $E$. So in this case, we shall remove the superscript $p$ from the notation.

Let $\Omega_C$ denote the sheaf of differentials on $C$. 
\begin{defn}
A \underline{parabolic Higgs bundle} is a pair $(E,\phi)$ consisting of a parabolic bundle $E$ and a strongly parabolic map $ \phi:E\to E\otimes \Omega_C(D)$.
\end{defn}

Let $E$ be a parabolic bundle over $C$ has a weight type $(\underline{s},\underline{w})$ and weight multiplicity $\underline{m}$. We denote :
\begin{itemize}
\item $\pd(E):=\deg(E)+\sum_{p\in D}\sum_{i=1}^{s_p}m_{p,i}w_{p,i}$ the parabolic degree of $E$,

\item $\pas(E):=\frac{\pd(E)}{\rk(E)}$ the parabolic slope of $E$.
\end{itemize}

\begin{defn}\begin{itemize}
\item We call the parabolic bundle $E$ stable (resp. semi-stable) if , for every proper subbundle $F$ of $E$, we have $\pas(F)<\pas(E)$ (resp. $\pas(F)\leq\pas(E)$).
\item We call a parabolic Higgs bundle $(E,\phi)$ stable (semi-stable) if the above inequalities hold on those proper subbundles $F$ of $E$, which are, in addition, $\phi$-invariant (i.e $\phi(F)\subset F\otimes \Omega_C(D)$).
\end{itemize}
\begin{rmk}
\begin{itemize}
\item $\pas(E\otimes^P F)=\pas(E)+\pas(F)$.
\item The dual parabolic bundle of $E$ is $E^*=\h(E,\mathcal{O}_C(-D))$ where the filtration at each point $p\in D$ is
    $$E^*_{p}=E^*_{p,1}\supset\dots \supset E^*_{p,s_p}\supset 0,$$
    with $E^*_{p,i}=\h(E_{p}/E_{p,s_p+2-i}, \mathcal{O}(-D)_{p})$ and the weights are $1-w_{p,s_p}<\dots<1-w_{p,1}$.
\item $E^{**}=E$ and $\pd(E^*)=-\pd(E)$.
\end{itemize}
\end{rmk}
\end{defn}
Likewise the non-parabolic case, any non-zero parabolic bundle admits the parabolic Harder-Narasimhan. More precisely, we have the following :
\begin{prop}
Any non-zero parabolic bundle $E$ admits a unique filtration by sub-bundles
$$0=G_0\varsubsetneq G_1 \varsubsetneq \dots \varsubsetneq G_h=E$$
satisfying
\begin{itemize}
\item $G_i/G_{i-1}$ is parabolic semi-stable for $i=1,\dots, h$,
\item $\pas(G_i/G_{i-1})>\pas(G_{i+1}/G_i)$ for $i=1,\dots,h-1$.
\end{itemize}
Equivalently,
\begin{itemize}
\item $G_i/G_{i-1}$ is parabolic semi-stable for $i=1,\dots, h$,
\item For any parabolic sub-bundle $E'$ of $E$ containing $G_{i-1}$ we have \linebreak $\pas(G_i/G_{i-1})>\pas(E'/G_{i-1})$, for $i=1,\dots,h$.
\end{itemize}
\end{prop}

And, the duality plays also an important role for parabolic Higgs bundles. The following results are consequences of theory developed by Yokogawa (cf. \cite[section 3.1 and proposition 3.7]{y1})
\begin{prop}\label{p5}
Let $E$ and $E'$ parabolic bundles. The sheaves $\ph(E',E)$ and $\sph(E,E'(D))$ are naturally dual.
\end{prop}
Let us fix the weight type $(\underline{s},\underline{w})$ and the weight multiplicity $\underline{m}$. We shall denote by $\mathcal{M}^{d}_{n,\mathcal{D}}$ the moduli stack of parabolic Higgs bundles of rank $n$, degree $d$ and of data $\mathcal{D}:=(\underline{s},\underline{w},\underline{m})$ on $C$ (the rank and the data must be compatible, i.e $\sum_{i=1}^{s_p}m_{p,i}=n\,\forall p\in D$).
As stability and semi-stability are open conditions, so stable parabolic Higgs bundles (resp. semi-stable parabolic Higgs bundles) define an open sub-stack $\mathcal{M}_{n,\mathcal{D}}^{d,\mathrm{st}}\subset(\text{resp. }\mathcal{M}_{n,\mathcal{D}}^{d,\mathrm{ss}})\subset \mathcal{M}_n^{d,\mathcal{D}}$.

We shall say that the datum $\mathcal{D}$ are $\textit{generic}$ when every semi-stable parabolic Higgs bundle is automatically stable. For the fixed generic data, the coarse moduli space of $\mathcal{M}_{n,\mathcal{D}}^{d,\mathrm{ss}}$ which we shall denote by $M_{n,\mathcal{D}}^{d}$ was constructed using Gauge Theory by Yokogawa (cf. \cite{y},\cite{y1}), who also showed that it is a smooth irreducible complex variety of dimension
$$2(g-1)n^2+2+\sum_{p\in D}\left(n^2-\sum_{i=1}^{s_p}m_{p,i}^2\right).$$

\begin{rmk}
If the degree $d$ and rank $n$ are coprime and all the weights are assumed to be very small ($w_{p,i}<\frac{1}{n^2.|D|}$ for example) then, each semi-stable parabolic Higgs bundle is actually stable parabolic.  
\end{rmk}

The following result is one of the criteria to test our computation of motive.
\begin{prop}\label{p3}(cf.\cite[proposition 2.1]{ggm}) Fix the rank $n$. For different choices of degree and generic datum, the $E$-polynomial of the moduli space $M_{n,\mathcal{D}}^{d}$ is the same.
\end{prop}
\begin{proof}
For fixed degree, it is a consequence of the results of Thaddeus \cite{t}(see. \cite[proposition 2.1]{ggm}  for brief explaining) and the property of $E$-polynomial that the $E$-polynomial of a projective bundle splits as the product of the $E$-polynomial of the base and the $E$-polynomial of projective space.

This result will be extended to moduli space with different degrees, by tensoring with a fixed parabolic line bundle $L$. By choosing a suitable parabolic line bundle and (generic) weight data, we obtain the proof. 
\end{proof}
The moduli space $M_{n,\mathcal{D}}^{d}$ has an action of $\mathbb{G}_m$, given by multiplication of scalars on the Higgs field $\phi$, i.e $\lambda.(E,\phi)=(E,\lambda\phi)$.
The fixed point set of the action of $\mathbb{G}_m$ is identified by Simpson and is analogous to what happens for non-parabolic Higgs bundles.
\begin{prop}[\cite{s}, theorem 8]\label{p2}
The equivalence class of a stable parabolic Higgs bundles $(E,\phi)$ is fixed under the action of $\mathbb{G}_m$ if and only if $E$ has a direct sum decomposition $E=\bigoplus_{i=0}^rE_i$ as parabolic bundles, such that the restriction $\phi_i:=\phi|_{E_i}$ of $\phi$ is a strongly parabolic map $E_i\to E_{i-1}\otimes\Omega_C(D)$. Furthermore, stability implies that $\phi|_{E_i}\neq 0$ for $i=1,\dots,r$.
\end{prop}
\begin{rmk}\begin{itemize}
\item If $m=0$, then $E=E_0$ and $\phi=0$, corresponding to the obvious fixed point $(E,0)$, with $E$ a stable parabolic bundle.
\item With the notation as in the above proposition, we have that $(E=\bigoplus_{i=0}^rE_i,\bigoplus_{i=1}^r\phi_i)$ is stable as a parabolic Higgs bundle if and only if the stability condition is satisfied for sub-bundles of $E$ with respect to the decomposition $E=\bigoplus_{i=0}^rE_i$.
\end{itemize}
\end{rmk}

Finally, as in the non-parabolic case, we have to recall the Hitchin map
$$f: \mathcal{M}^{d}_{n,\mathcal{D}} \to \mathcal{A}:=\oplus_{i=1}^nH^0(C,\Omega_C(D)^{\otimes i})$$ given by taking the parabolic Higgs bundle $(E,\phi)$ to the characteristic polynomial of $\phi$. Using a method developed by Langton (\cite{l}), Yokogawa shows in Corollary 5.13 and 1.6 of \cite{y} that the Hitchin map induces proper map on the moduli space $M_{n,\mathcal{D}}^{d}$.
\begin{thm}
Let ${M}^d_{n,\mathcal{D}}$ be a moduli space of semi-stable parabolic Higgs bundles of a fixed type. The the Hitchin map $M^d_{n,\mathcal{D}}\to \mathcal{A}$ is proper.
\end{thm}
The Hitchin map $f$ becomes equivariant with respect to the action of $\mathbb{G}_m$, if we let $\mathbb{G}_m$ act by the character $\lambda\to \lambda^i$ on the sub-space $H^0(C,\Omega_C(D)^{\otimes i})\subset \mathcal{A}$.

As in the case non-parabolic (\cite{ghs}), we have also the properties of the Bialynicki-Birula decomposition (\cite{bb}) with respect to this action which we shall recollect in the following proposition :
\begin{prop}\label{p1}
Let $n,d$ be a fixed pair of positive integers.
\begin{itemize}
\item The fixed point scheme $(M_{n,\mathcal{D}}^{d})^{\mathbb{G}_m}$ of the $\mathbb{G}_m$ action on $M_{n,\mathcal{D}}^{d}$ is a disjoint union of connected, smooth schemes $F_i$ contained in the special fiber $f^{-1}(0)$ of the Hitchin map.
\item There are $\mathbb{G}_m$-subvarieties $F_i^+,\, F_i^-\subset M_{n,\mathcal{D}}^d$ such that $F_i$ is a closed sub-scheme of $F_i^{\pm}$ and $F_i^{\pm}$ is a Zariski-locally trivial fibration over $F_i$, with fibres isomorphic to affine spaces. For any $x\in F_i$ we have $T_x(F_i^{\pm})=T_x(M)^0\oplus T_x(M)^{\pm}$ where $T_x(M)^{0,+,-}$ are the weight spaces of the tangent
space at $x$ with weight respectively 0, positive or negative.
\item We have $M^{d}_{n,\mathcal{D}}=\cup F_i^+$ and $f^{-1}(0)=\cup F_i^-$ and for all points $x\in (M^{d}_{n,\mathcal{D}})^{\mathbb{G}_m}$ we have $\dim(T_x(M)^+)=\frac{1}{2}\dim(M^{d,\mathcal{D}}_n)$. In particular the closure of the $F_i^-$ in $M^{d}_{n,\mathcal{D}}$ are the irreducible components of $f^{-1}(0)$.
\end{itemize}
\end{prop}
\begin{proof}
The first part of this proposition is a direct consequence of the Bialynicki-Birula decomposition theorem (\cite[Theorem 4.1]{bb} for algebraically closed field $k$ and \cite[Theorem 5.8]{h} in general). The varieties $F_i^+$ consist of those points such that $\lim_{t\to 0}t.x\in F_i$ and $F_i^-$ consists of the points such that $\lim_{t\to\infty}t.x\in F_i$.

Next we use the result of Yokogawa which the Hitchin map is proper and the fact that $\mathbb{G}_m$ acts with positive weights on $\mathcal{A}$.

The dimension formula is given in \cite[Corrolary 3.10]{ggm}.
\end{proof}
\begin{rmk}
In the non-parabolic case, Hausel (\cite[Theorem 5.2]{th}) proved that the downwards Morse flow on the moduli space of Higgs bundles coincides with nilpotent cone (the pre-image of $0$ under the Hitchin map, i.e $f^{-1}(0)$) who is Lagrangian according to Laumon (\cite{gl}). The proof of Hausel goes over word by word to the parabolic case and the nilpotent cone is isotropic because the Hitchin map is a completely integrable system. Moreover, the dimension of nilpotent cone equals haft of the moduli space $M_n^{d,\mathcal{D}}$ (cf. third part of proposition \ref{p1}), so we obtain the parabolic version of Laumon's theorem:
\begin{thm}
The nilpotent cone is a Lagrangian subvariety of the moduli space of parabolic Higgs bundles.
\end{thm}
\end{rmk}

The proposition \ref{p1} implies that the class of $[M^{d}_{n,\mathcal{D}}]\in \hat{K_0}(\mathrm{Var})$ can be computed in a very simple way from the classes of the $F_i$. So we obtain the corollary following likewise observed in \cite[Proposition 9.1]{ht} in the non-parabolic case :
\begin{cor}\label{corm}
Write $$N:=\frac{1}{2}\dim(M^d_{n,\mathcal{D}})= n^2(g-1)+1+\frac{1}{2}\sum_{p\in D}\left(n^2-\sum_{i=1}^{s_p}(m_i(p))^2\right)$$ then we have
$$[M^{d}_{n,\mathcal{D}}]=\mathbb{L}^N\sum_{i}[F_i]\in \hat{K_0}(\mathrm{Var}).$$
\end{cor}
Using this result and the proposition \ref{p2}, we have a relation between the class of moduli space of parabolic Higgs bundles and the classes of $(\bigoplus_{i}E_i,\bigoplus_i\phi_i)$ (satisfying the conditions in proposition \ref{p2}) which we shall call \textit{moduli spaces of stable parabolic chains}.
One of the main part of this paper will be devoted to the computation of the classes of these moduli spaces.
\section{Parabolic chains}
\subsection{Definitions and basic facts}
A \textit{(holomorphic) parabolic chain} on $C$ of length $r$ is a collection $\mathbf{E_\bullet^r}=((E_i)_{i=0,\dots,r},(\phi_i)_{i=1,\dots,r})$, where $E_i$ are parabolic vector bundles on $C$ and $\phi_i: E_i\to E_{i-1}$ are strongly parabolic morphisms.

A homomorphism from $\mathbf{{E'}_\bullet^{r}}$ to $\mathbf{E_\bullet^r}$ is a collection of commutative diagrams
$$\xymatrix{E'_i\ar[r]^-{\phi'_i}\ar[d]&E'_{i-1}(D)\ar[d]\\
E_i\ar[r]^-{\phi_i}&E_{i-1}(D)
}$$
where the vertical arrows are parabolic homomorphisms. A parabolic chain $\mathbf{{E'}_\bullet^r}$ is a sub-chain of $\mathbf{E_\bullet^r}$ if the sheaf homomorphisms $E'_i\to E_i$ are injective. A sub-chain $ \mathbf{{E'}_\bullet^r}\subset \mathbf{{E}_\bullet^r}$ is called \textit{proper} if $\mathbf{{E'}_\bullet^r}\neq 0$ and $\mathbf{{E'}_\bullet^r}\neq \mathbf{{E}_\bullet^r}$.

\begin{defn}
The rank of $\mathbf{{E}_\bullet^r}$ is defined as $\rk(\mathbf{{E}_\bullet^r})=(\rk(E_i))_{i=0,\dots,r}$ and the degree is defined as $\deg(\mathbf{{E}_\bullet^r})=(\deg(E_i))_{i=0,\dots,r}$. For $\underline{\alpha}=(\alpha_i)_{i=0,\dots,r}\in \mathbb{R}^{r+1}$ the $\underline{\alpha}$-slope of $\mathbf{{E}_\bullet^r}$ is defined as
$$\pas_{\underline{\alpha}}(\mathbf{{E}_\bullet^r}):=\sum_{i=0}^r\frac{\rk(E_i)}{|\rk(\mathbf{{E}_\bullet^r})|}(\pas(E_i)+\alpha_i).$$
\end{defn}

We say $\mathbf{{E}_\bullet^r}$ is \textit{$\alpha$-(semi)-stable} if
$$\pas_{\underline{\alpha}}(\mathbf{{E'}_\bullet^r}) (\leq)< \pas_{\underline{\alpha}}(\mathbf{{E}_\bullet^r})$$ for any proper sub-chain $\mathbf{{E'}_\bullet^r}$ of $\mathbf{{E}_\bullet^r}$. 

Note that for any $c\in \mathbb{R}$, $\underline{\alpha}=(\alpha_i)_{i=0,\dots,r}$ and $\underline{\alpha}+c:=(\alpha_i+c)_{i=0,\dots,r}$ define the same (semi)-stability condition.

Given a parabolic chain $\mathbf{E_\bullet^r}=((E_i)_{i=0,\dots,r},(\phi_i)_{i=1,\dots,r})$ one has the dual parabolic chain $\mathbf{E_\bullet^r}^*=((E_i^*)_{i=r,\dots,0},(\phi_i^*)_{i=r,\dots,1})$, where $E_i^*$ is the parabolic dual of $E_i$ and $\phi_i^*$ is the transpose of $\phi_i$. We have the following proposition.
\begin{prop}
Noting that $\overline{\alpha}:=(-\alpha_{r-i})_{i=0,\dots,r}$. The $\underline{\alpha}$-(semi)-stability of $\mathbf{E_\bullet^r}$ is equivalent to the $\overline{\alpha}$-(semi)-stability of $\mathbf{E_\bullet^r}^*$. In particular for $\alpha_i=(2g-2)i$, the $\underline{\alpha}$-(semi)-stability of $\mathbf{E_\bullet^r}$ is equivalent to the $\underline{\alpha}$-(semi)-stability of $\mathbf{E_\bullet^r}^*$. The map $\mathbf{E_\bullet^r}\mapsto \mathbf{E_\bullet^r}^*$ defines an isomorphism of moduli spaces.
\end{prop}
\begin{proof}
The first claim is immediate from the definition of $\underline{\alpha}$-(semi)-stability. The second follows, because in this case $\overline{\alpha}+(2g-2)r=\underline{\alpha}$.
\end{proof}
The parabolic chains are related to the parabolic Higgs bundles by the proposition following.
\begin{prop}
Suppose that $(E,\Phi)$ is a (semi)-stable parabolic Higgs bundle such that $E=\bigoplus_{i=0}^r E_i$ where $E_i$ are parabolic bundles and
$$\Phi:=\left(\begin{matrix}0&\phi_1&0&\cdots&0\\
\vdots&\ddots&\ddots&\ddots&\vdots\\
\vdots&&\ddots&\ddots&0\\
0&\cdots&\cdots&0&\phi_r\\
0&\cdots&\cdots&\cdots&0
\end{matrix}\right)$$ with $\phi_i: E_i\to E_{i-1}\otimes\Omega_C(D)$ a strongly parabolic map. Then $(E,\Phi)$ is (semi)-stable if and only if the parabolic chain $((E_i\otimes \Omega^{r-i}_C)_{i=0,\dots,r},(\phi_i)_{i=0,\dots,r})$ is $\underline{\alpha}$-(semi)-stable for $\underline{\alpha}=(0,2g-2,\dots,r(2g-2))$.
\end{prop}
\begin{proof} We consider a sub-bundle $E'\subset E$ with $\Phi(E')\subset E'\otimes \Omega_C(D)$. We can rewrite $E'=\oplus_{i=0}^rE'_i$ by taking $E'_i:=E_i\cap E'$ and hence it defines a sub-chain $((E'_i\otimes \Omega_C^{r-i})_{i=0,\dots,r},(\phi'_i)_{i=0,\dots,r})$ where $\phi'_i=\phi_i|_{E'_i}$. The result follows now from the equivalence between
$$ \sum_{i=0}^{r}\frac{\rk(E'_i)}{|\rk(\mathbf{E'^r_\bullet})|}\pas(E'_i)(\leq)<\sum_{i=0}^{r}\frac{\rk(E_i)}{|\rk(\mathbf{{E}_\bullet^r})|}\pas(E_i),\quad \text{and}$$
\begin{eqnarray*}\sum_{i=0}^r\frac{\rk(E'_i)}{|\rk(\mathbf{E'^r_\bullet})|}(\pas(E'_i)+(r-i)(2g-2)+i(2g-2))\\(\leq)<\sum_{i=0}^r\frac{\rk(E_i)}{|\rk(\mathbf{{E}_\bullet^r})|}(\pas(E_i)+(r-i)(2g-2)+i(2g-2)) \end{eqnarray*}
Here, we use $\pas(E^{\circ}_i)+(r-i)(2g-2)=\pas(E^{\circ}_i\otimes\Omega_C^{r-i})$ (where we can replace $\circ$ by $'$ or nothing) so the second inequality above is the $\underline{\alpha}$-(semi)-stability of the chain $((E_i\otimes \Omega^{r-i}_C)_{i=0,\dots,r},(\phi_i)_{i=0,\dots,r})$, for $\underline{\alpha}=(0,2g-2,\dots,r(2g-2))$.
\end{proof}
Let $\mathbf{E'^r_{\bullet}}$ and $\mathbf{E''^r_{\bullet}}$ be two parabolic chains. Let $\ho(\mathbf{E''^r_{\bullet}},\mathbf{E'^r_{\bullet}})$ denote the linear space of homomorphism from $\mathbf{E''^r_{\bullet}}$ to $\mathbf{E'^r_{\bullet}}$ where by this we mean a collection of commutative diagrams
$$\xymatrix{E''_i\ar[r]^-{\phi''_i}\ar[d]&E''_{i-1}(D)\ar[d]\\
E'_i\ar[r]^-{\phi'_i}&E'_{i-1}(D)
}$$ with the vertical arrows are parabolic homomorphisms, and let $\e^1(\mathbf{E''^r_{\bullet}},\mathbf{E'^r_{\bullet}})$ denote the linear space of equivalence classes of extensions of the form
$$0\to \mathbf{E'^r_{\bullet}}\to \mathbf{E^r_{\bullet}}\to \mathbf{E''^r_{\bullet}}\to 0,$$
where by this we mean a collection of commutative diagrams
$$\xymatrix{0\ar[r]&E'_i\ar[d]^-{\phi'_i}\ar[r]&E_{i}\ar[d]^-{\phi_i}\ar[r]&E''_{i}\ar[d]^-{\phi''_i}\ar[r]&0\\
0\ar[r]&E'_{i-1}(D)\ar[r]&E_{i-1}(D)\ar[r]&E''_{i-1}(D)\ar[r]&0
}$$ with the horizontal arrows are parabolic homomorphisms.

As in the case of parabolic bundles, given an extension $\mathbf{E'^r_{\bullet}}\to \mathbf{E^r_{\bullet}}\to \mathbf{E''^r_{\bullet}}$
of chains of rank $\underline{n}'$ and $\underline{n}''$ we have
$$\pas_{\underline{\alpha}}(\mathbf{E^r_\bullet})=\frac{|\underline{n}'|}{|\underline{n}|}\pas_{\underline{\alpha}}(\mathbf{E'^r_\bullet})+\frac{|\underline{n}''|}{|\underline{n}|}\pas_{\underline{\alpha}}(\mathbf{E''^r_\bullet}).$$
So the slope of an extension is a convex combination of the slope of the constituents. This property immediately implies the following properties of stability parabolic chains:
\begin{lem}\label{lem3}\begin{enumerate}
\item A parabolic chain $\mathbf{E^r_\bullet}$ is $\underline{\alpha}$-semi-stable if and only if for any quotient $\mathbf{E^r_{\bullet}}\twoheadrightarrow \mathbf{E''^r_\bullet}$ we have $\pas_{\underline{\alpha}}(\mathbf{E^r_{\bullet}})\leq\pas_{\underline{\alpha}}(\mathbf{E''^r_{\bullet}})$.

\item If $\mathbf{E^r_{\bullet}},\mathbf{F^r_\bullet}$ are $\underline{\alpha}$-semi-stable with $\pas_{\underline{\alpha}}(\mathbf{E^r_{\bullet}})>\pas_{\underline{\alpha}}(\mathbf{F^r_{\bullet}})$ then \linebreak $\ho(\mathbf{E^r_{\bullet}},\mathbf{F^r_{\bullet}})=0$.
\item For every parabolic chain $\mathbf{E^r_{\bullet}}$ there is a canonical Harder-Narasimhan flag of parabolic sub-chains $0\subset\mathbf{E^r_{\bullet}}^{(1)}\subset\dots\subset\mathbf{E^r_{\bullet}}^{(h)}=\mathbf{E^r_{\bullet}}$, such that
$\pas_{\underline{\alpha}}(\mathbf{E^r_{\bullet}}^{(1)})>\dots>\pas_{\underline{\alpha}}(\mathbf{E^r_{\bullet}}^{(h)})$ and the sub-quotients $\mathbf{E^r_{\bullet}}^{(i)}/\mathbf{E^r_{\bullet}}^{(i-1)}$ are $\underline{\alpha}$-semi-stable. 
\end{enumerate}
\end{lem}

For any parabolic chain $\bf{E^r_\bullet}$ we shall denote by $\pas_{\underline{\alpha},\max}(\mathbf{E^r_\bullet}):=\pas_{\underline{\alpha}}(\mathbf{E^r_\bullet}^{(1)})$ the maximal $\underline{\alpha}$-slope of parabolic sub-chains of $\mathbf{E^r_\bullet}$ and by
$\pas_{\underline{\alpha},\min}(\mathbf{E^r_\bullet}):=\pas_{\underline{\alpha}}(\mathbf{E^r_\bullet}^{(h)}/\mathbf{E^r_\bullet}^{(h-1)})$ the minimal $\underline{\alpha}$-slope of quotients of $\bf{E}_\bullet$.
\subsection{Necessary conditions for the existence of semi-stable parabolic chains}
In this subsection we want to collect conditions on $\underline{n},\underline{d},\underline{\alpha}$ which are necessary for the existence of $\underline{\alpha}$-semi-stable parabolic chains of rank $\underline{n}$, degree $\underline{d}$ and fixed generic data $\mathcal{D}$. We can find this kind of necessary conditions in \cite[Proposition 4]{gh} for non-parabolic chains and in \cite[Proposition 4.3]{ggm} for parabolic triples.
\begin{prop}\label{p4}
Fix $\underline{n}=(n_0,\dots,n_r)\in \mathbb{N}^{r+1}$, $\underline{d}=(d_0,\dots,d_r)\in \mathbb{Z}^{r+1}$, $\underline{\alpha}=(\alpha_0,\dots,\alpha_r)\in \mathbb{R}^{r+1}$ a (semi)-stability parameter satisfying $\alpha_0<\alpha_1<\dots<\alpha_r$ and a (generic) data $\mathcal{D}$.
An $\underline{\alpha}$-semi-stable parabolic chain $\bf{E^r_{\bullet}}$ of rank $\underline{n}$ and degree $\underline{d}$ can only exist if
\begin{enumerate}
\item for all $j\in {0,\dots,r-1}$, we have
$$\sum_{i=0}^j\frac{n_i}{\sum_{k=0}^jn_k}(\pas(E_i)+\alpha_i)\leq \pas_{\underline{\alpha}}(\bf{E^r_{\bullet}});
$$
\item for all j such that $n_j=n_{j-1}$, we have
$$\pd(E_j)-n_j.|D|\leq \pd(E_{j-1});$$
\item for all $0\leq k<j\leq r$ such that $n_j<\min\{n_k,\dots,n_{j-1}\}$, we have
\begin{eqnarray*}
\frac{\sum_{i\not\in[k,j]}(\pd(E_i)+n_i\alpha_i)+(j-k+1)\pd(E_j)}{\sum_{i\not\in[k,j]}n_i+(j-k+1)n_j}+\\
+\frac{\left(\sum_{i=k}^j\alpha_i-\frac{(j-k+1)(j-k)}{2}s\right)n_j}{\sum_{i\not\in[k,j]}n_i+(j-k+1)n_j}
\leq\pas_{\underline{\alpha}}(\bf{E^r_{\bullet}});
\end{eqnarray*}

\item for all $0\leq k<j\leq r$ such that $n_k<\min\{n_{k+1},\dots,n_{j}\}$, we have
\begin{eqnarray*}\frac{\sum_{i=k+1}^j(\pd(E_i)-\pd(E_k)-n_k(i-k)|D|+\alpha_i(n_i-n_k))}{\sum_{i=k+1}^j(n_i-n_k)}\\\leq \pas_{\underline{\alpha}}(\bf{E^r_{\bullet}}).
\end{eqnarray*}
\end{enumerate}
\end{prop}
\begin{proof}The proof is similar to the one given in \cite[Proposition 4]{gh}.
\end{proof}
\begin{rmk}
\begin{itemize}
\item The inequalities (1) simply come from that \linebreak $(E_0,\dots,E_j,0,\dots,0)$ is the obvious sub-chains of $\bf{E^r_\bullet}$.
\item Given chain $\bf{E^r_\bullet}$, condition (3) is equivalent to
$$\pas_{\underline{\alpha}}(\bf{E'^r_\bullet})\leq \pas_{\underline{\alpha}}(\bf{E^r_\bullet}),$$
where the chain $\bf{E'^r_\bullet}$ is obtained from $\bf{E^r_\bullet}$ by replacing the parabolic bundles $E_i$ by $E_j(-(j-i)D)$ for $i=j-1,\dots, k$. However, this chain is a sub-chain of $\bf{E^r_\bullet}$ only if the composition of the maps
$E_j(-(j-k)D)\to E_{j-1}(-(j-1-k)D)\to\dots\to E_k$ is injective.
\item Similarly, (4) expresses the inequality
$$\pas_{\underline{\alpha}}(\bf{E^r_\bullet})\leq \pas_{\underline{\alpha}}(\bf{E''^r_\bullet}),$$
where the chain $\bf{E''^r_\bullet}$ is obtained from $\bf{E^r_\bullet}$ by replacing the parabolic bundles $E_i$ by $E_k((i-k)D)$ for $i=k+1,\dots, j$. This chain is a quotient chain of $\bf{E^r_\bullet}$ only if the composition of the maps $E_j\to E_{j-1}(D)\to\dots\to E_k((j-k)D)$ is surjective. We can see that the last condition (4) is a dual of the condition (3) by mean of passing to the dual chain. So the condition (4) of \cite[Proposition 4]{gh} must be $n_k<\min\{n_{k+1},\dots,n_j\}$ instead of $n_j>\max\{n_k,\dots,n_{j-1}\}$.
\end{itemize}
\end{rmk}
For fixed data, the parabolic degree depends only on the degree so we can see the conditions of proposition \ref{p4} as the conditions for $d_i$. Similarly \cite[Corollary 6]{gh}, we obtain the corollary following.
\begin{cor}\label{cor1}
For fixed data, fixed values of $\underline{n}\in\mathbb{N}^{r+1}$ and $\underline{\alpha}\in\mathbb{R}^{r+1}$ satisfying $\alpha_r>\alpha_{r-1}>\dots>\alpha_0$ and $d\in \mathbb{Z}$, there are only finitely many values $\underline{d}\in \mathbb{Z}^{r+1}$ with $\sum d_i=d$ which satisfy the condition given in Proposition \ref{p4}.
\end{cor}

\subsection{Extensions and deformations of parabolic chains}
In this section for the convenience, we shall consider the more generally parabolic chains $\mathbf{E}_{\bullet}=(E_i,\phi_i)$ with $i\in \mathbb{Z}$ where there are only finite many $E_i$ are non-zero. We shall extend any chain $\mathbf{E^r_{\bullet}}$ by putting $E_i:=0$ for all $i<0$ and all $i>r$. Similarly we shall allow (semi)-stability parameters $\underline{\alpha}=(\alpha_i)_{i\in \mathbb{Z}}$.

Now we consider the following complex of sheaves
$$C^\bullet(\mathbf{E''}_\bullet,\mathbf{E'}_\bullet):\oplus_{i}\ph(E''_i,E'_i)\buildrel c\over{\rightarrow} \oplus_{i}\sph(E''_i,E'_{i-1}(D))$$ where the differential $c$ is defined by
$$c((f_i))=(\phi'_i\circ f_i - f_{i-1}\circ \phi''_i).$$
\begin{prop}\label{p6}
There are natural isomorphisms
$$\ho(\mathbf{E''}_\bullet,\mathbf{E'}_\bullet)\simeq \mathbb{H}^0(C^\bullet(\mathbf{E''}_\bullet,\mathbf{E'}_\bullet)),$$
$$\e^1(\mathbf{E''}_\bullet,\mathbf{E'}_\bullet)\simeq \mathbb{H}^1(C^\bullet(\mathbf{E''}_\bullet,\mathbf{E'}_\bullet)),$$
and a long exact sequence associated to the complex $C^\bullet(\mathbf{E''}_\bullet,\mathbf{E'}_\bullet)$ :
$$\begin{matrix}
0&\to&\mathbb{H}^0&\to&H^0(\oplus_i\ph(E''_i,E'_i))&\to&H^0(\oplus_i\sph(E''_i,E'_{i-1}(D)) \\
&\to&\mathbb{H}^1&\to&H^1(\oplus_i\ph(E''_i,E'_i))&\to&H^1(\oplus_i\sph(E''_i,E'_{i-1}(D))\\
&\to&\mathbb{H}^2&\to &0.&
\end{matrix}$$
where $\mathbb{H}^i=\mathbb{H}^i(C^\bullet(\mathbf{E''}_\bullet,\mathbf{E'}_\bullet)).$

Moreover, if the $\phi ''_i$ are injective for all $i$ or the $\phi'_i$ are generically surjective for all $i$, then $\mathbb{H}^2=0.$
\end{prop}
\begin{proof}
The proof is omitted since it is a special case of a much more general result proved in \cite[Theorem 4.1 and 5.1]{gk}.
\end{proof}
The above proposition is most useful, if the last group $\mathbb{H}^2$ in the above sequence vanishes. Now, we shall see how to handle it.

Applying Serre duality for hyper-cohomology of $C^\bullet(\mathbf{E''}_\bullet,\mathbf{E'}_\bullet)$ we have :
\begin{equation}\label{eqn2}
\mathbb{H}^i(C^\bullet(\mathbf{E''}_\bullet,\mathbf{E'}_\bullet))^*\cong \mathbb{H}^{2-i}(C^\bullet(\mathbf{E'}_\bullet,\mathbf{E''}_{\bullet-1}\otimes \Omega_C)),
\end{equation}
where $\mathbf{E''}_{\bullet-1}$ is the parabolic chain obtained by shifting the parabolic chain $\mathbf{E''}_{\bullet}$ by placing $E''_i$ in the degree $i-1$, so that the bundles of the resulting chains may be non-zero for $-1\leq i\leq r-1$. Here, we used the Proposition \ref{p5}.

Take $i=2$ in the Isomorphism \ref{eqn2}, we obtain the lemma following.
\begin{lem}\label{lem1}
Let $\mathbf{E'}_{\bullet},\mathbf{E''}_{\bullet}$ be parabolic chains, then we have
$$\mathbb{H}^2(C^\bullet(\mathbf{E''}_\bullet,\mathbf{E'}_\bullet))^*\cong \ho(\mathbf{E''}_\bullet,\mathbf{E'}_{\bullet-1}).$$
\end{lem}

Using the above lemma, we have got the lemma following which is the key to our computation
\begin{lem}\label{keylem}
Suppose for all $i$ we have $\alpha_i-\alpha_{i-1}\geq 2g-2$.
\begin{enumerate}
\item Let $\mathbf{E'}_{\bullet},\mathbf{E''}_{\bullet}$ be parabolic chains of slope $\pas_{\underline{\alpha},\min}(\bf{E'_\bullet})>\pas_{\underline{\alpha},\max}(\bf{E''_\bullet})$. Then
$\mathbb{H}^2(C^\bullet(\mathbf{E''}_\bullet,\mathbf{E'}_\bullet))=0$.
\item If $\bf{E_\bullet}$ is an $\underline{\alpha}$-stable parabolic chain, then $\mathbb{H}^2(C^\bullet(\mathbf{E}_\bullet,\mathbf{E}_\bullet))=0$.
\end{enumerate}
\end{lem}
\begin{proof}
The proof is similar to the one of \cite[Lemma 4.6]{ghs}.
\end{proof}
\subsection{The classes of Harder-Narasimhan strata}

We note by $\pb_{n,\mathcal{D}}^d$ the moduli stack of parabolic bundles of rank $n$, degree $d$ and data $\mathcal{D}$, i.e : 

$$\pb_{n,\mathcal{D}}^d(T):=\left<\begin{matrix}E \in \mathrm{Bundle}_n^d(T),\\(E_{p,s_p}\subset \dots \subset E_{p,1}=E_{|p\times T})_{p\in D} \\
\text{ a partial flag of type } (n, s_p, (m_{p,i})).  
\end{matrix}\right>$$
We recall that $\mathrm{Bundle}_n^d$ is the moduli stack of vector bundles of rank $n$ and degree $d$; and a partial flag  $E_{p,s_p}\subset \dots \subset E_{p,1}$ is called of type $(n, s_p, (m_{p,i}))$, if $\dim(E_{p,i})-\dim(E_{p,i+1})=m_{p,i}$ (note that $\dim(E_{p,s_p+1})=0$). 

We denote by $\mathrm{Flag}(n,m,(r_i))$ the variety of all flags $k^n=F_1\supset \dots F_m\supset F_{m+1}=0$ of vector subspaces in $k^n$, with $\dim(F_i)-\dim(F_{i+1})=r_i$. By identifying this variety with the homogeneous space $\mathrm{GL}_n/P$, where  $P$
to be the group of nonsingular block upper triangular matrices, where the dimensions of the blocks are $r_i$, we obtain the following lemma :
\begin{lem}\label{flag1}
\begin{eqnarray*}
[\mathrm{Flag}(n,m,(r_i))]&=&\frac{[\mathrm{GL}_n]}{[P]}\\
&=& \frac{\prod_{i=0}^{n-1}(\mathbb{L}^n-\mathbb{L}^i)}{\left[\prod_{i=1}^m\prod_{j=0}^{r_i-1}(\mathbb{L}^{r_i}-\mathbb{L}^j)\right]\left[\mathbb{L}^{\sum_{i=1}^{m-1}r_i\sum_{j=i+1}^mr_j}\right]}\\
&=& \frac{\prod_{i=0}^{n-1}(\mathbb{L}^n-\mathbb{L}^i)}{\left[\prod_{i=1}^m\prod_{j=0}^{r_i-1}(\mathbb{L}^{r_i}-\mathbb{L}^j)\right]\left[\mathbb{L}^{\sum_{1\leq i <j\leq m}r_ir_j}\right]}
\end{eqnarray*}
\end{lem}
\begin{proof}
To see $[\mathrm{GL}_n]=\prod_{i=0}^{n-1}(\mathbb{L}^n-\mathbb{L}^i)$, we use the usual argument that the first column of an invertible matrix is an arbitrary element of $\mathbb{A}^n-\{0\}$, the second then gives a factor $\mathbb{A}^n-\mathbb{A}^1$ and so on.

To calculate $[P]$, we can see that $$[P]=\prod_{i=1}^m[\mathrm{GL}_{r_i-r_{i+1}}]\prod_{i=1}^{m-1}[\mathrm{Mat}_{r_i\times \sum_{j=i+1}^mr_j}],$$ where $\mathrm{Mat}_{r_i\times \sum_{j=i+1}^mr_j}$ is an affine space of matrices of type ${r_i\times \sum_{j=i+1}^mr_j}$.
\end{proof}

Consider the forgetful map $\pb^{d}_{n,,\mathcal{D}}\to \mathrm{Bundle}_n^d$. The fibres of this morphism are products of flags manifolds $\prod_{p\in D}\mathrm{Flag}(n,s_p,(m_{p,i}))$. Using the lemma above, we have 
\begin{cor}
{\small\begin{eqnarray*}[\pb_{n,\mathcal{D}}^d]&=&[\mathrm{Bundle}^d_n]\times\\
&\times &\prod_{p\in D}\frac{\prod_{i=0}^{n-1}(\mathbb{L}^n-\mathbb{L}^i)}{\left[\prod_{i=1}^{s_p}\prod_{j=0}^{m_{p,i}-1}(\mathbb{L}^{m_{p,i}}-\mathbb{L}^j)\right]\left[\mathbb{L}^{\sum_{1\leq i<j\leq s_p}m_{p,i}m_{p,j}}\right]}\end{eqnarray*}}
\end{cor}

\begin{rmk}
Using the argument of \cite{bgl}, Behrend and Dhillon \cite{bd} show that 
$$[\mathrm{Bundle}^d_n]=\mathbb{L}^{(n^2-1)(g-1)}\frac{[\mathrm{Pic}^0]}{\mathbb{L}-1}\prod_{i=2}^nZ(C,\mathbb{L}^{-i}),$$
where $Z(C,t):=\sum_{i\geq 0}[\mathrm{Sym}^i(C)]\mathbb{L}^i$ the zeta function of $C$ and $\mathrm{Pic}^0$ is the Jacobian of $C$.
\end{rmk}

We shall denote by $\pc^{r,\underline{d}}_{\underline{n},\underline{\mathcal{D}}}$ the moduli stack of parabolic chains of length $r$, rank $\underline{n}$, degree
$\underline{d}$ and data $\underline{\mathcal{D}}$. To show that the stack $\pc^{r,\underline{d}}_{\underline{n},\underline{\mathcal{D}}}$ is an algebraic stack, locally of finite type, we only have to 
observe that the forgetful map $\pc^{r,\underline{d}}_{\underline{n},\underline{\mathcal{D}}}\to \prod_{i=0}^r \pb_{n_i,\mathcal{D}_i}^{d_i}$ is representable. This holds, because the fibres parameterize morphisms of sheaves.

Given parabolic chains $\mathbf{E'_\bullet},\mathbf{E''_\bullet}$, we denote by $\underline{\e}(\mathbf{E''_\bullet},\mathbf{E'_\bullet})$ the stack of extensions $\mathbf{E'_\bullet}\to \mathbf{E_\bullet}\to \mathbf{E''_\bullet}$.

For parabolic chains $\mathbf{E_\bullet}^{(1)},\dots,\mathbf{E_\bullet}^{(h)}$ we denote by $\underline{\e}(\mathbf{E_\bullet}^{(h)},\dots,\mathbf{E_\bullet}^{(1)})$ the stack of iterated extensions, i.e., parabolic chains $\mathbf{E_\bullet}$ together with 
a filtration $0=\mathbf{F_\bullet}^{(0)}\subset\mathbf{F_\bullet}^{(1)}\subset\dots\subset\mathbf{F_\bullet}^{(h)}=\mathbf{E_\bullet}$ and parabolic isomorphisms $\mathbf{F_\bullet}^{(i)}/\mathbf{F_\bullet}^{(i-1)}\cong \mathbf{E_\bullet}^{(i)}$. Similarly, fixing given 
ranks $\underline{n}^i$, degrees $\underline{d}^i$ and datum $\underline{\mathcal{D}}^i$, we denote by $\underline{\e}(\underline{n}^h,\dots,\underline{n}^1)_{\underline{\mathcal{D}}^h,\dots,\underline{\mathcal{D}}^1}^{\underline{d}^h,\dots,\underline{d}^1}$ the stack 
of parabolic chains $\mathbf{E_\bullet}$ together with a filtration $\mathbf{F_\bullet}^{(i)}$ such that $$\left(\rk(\mathbf{F_\bullet}^{(i)}/\mathbf{F_\bullet}^{(i-1)}),\deg(\mathbf{F_\bullet}^{(i)}/\mathbf{F_\bullet}^{(i-1)})\right)=(\underline{n}^i,\underline{d}^i)$$ and the weights data of $\mathbf{F_\bullet}^{(i)}/\mathbf{F_\bullet}^{(i-1)}$ is $\underline{\mathcal{D}}^i$. Also 
we shall denote by 
$$\underline{\e}(\underline{n}^h,\dots,\underline{n}^1)_{\underline{\mathcal{D}}^h,\dots,\underline{\mathcal{D}}^1}^{\underline{d}^h,\dots,\underline{d}^1,\mathrm{gr}\, \underline{\alpha}-ss}\subset \underline{\e}(\underline{n}^h,\dots,\underline{n}^1)_{\underline{\mathcal{D}}^h,\dots,\underline{\mathcal{D}}^1}^{\underline{d}^h,\dots,\underline{d}^1}$$
the open sub-stack of filtered parabolic chains such that the sub-quotients $\mathbf{F_\bullet}^{(i)}/\mathbf{F_\bullet}^{(i-1)}$ are $\underline{\alpha}$-semi-stable.

Now we shall use the result of previous subsection to describe the Harder-Narasimhan strata of parabolic chains : 
\begin{prop}\label{p7}
Let $\underline{\alpha}$ be a semi-stability parameter and $(\underline{n}^i,\underline{d}^i,\underline{\mathcal{D}}^i)_{i=1,\dots,h}$ be ranks, degrees and weights datum of parabolic chains of length $r$. Suppose that 
$\alpha_j-\alpha_{j-1}\geq 2g-2$ for $j=1,\dots,r$ and $\pas_{\underline{\alpha}}(\underline{n}^i,\underline{d}^i,\underline{\mathcal{D}}^i)>\pas_{\underline{\alpha}}(\underline{n}^{i+1},\underline{d}^{i+1},\underline{\mathcal{D}}^{i+1})$ for $i=1,\dots,h-1$.
Then the forgetful map:
$$\mathrm{gr}: \underline{\e}(\underline{n}^h,\dots,\underline{n}^1)_{\underline{\mathcal{D}}^h,\dots,\underline{\mathcal{D}}^1}^{\underline{d}^h,\dots,\underline{d}^1,\mathrm{gr}\, \underline{\alpha}-ss}
\to \prod_{i=1}^h\pc_{\underline{n}^i,\underline{\mathcal{D}}^i}^{\underline{d}^i,\underline{\alpha}-ss}$$
is smooth and its fibers are affine spaces of dimension $\chi=\sum_{1\leq i<j\leq h}\chi_{ij}$, where 
\begin{multline*}
\chi_{ij}=\sum_{k=0}^r\left[n^j_k n^i_k(g-1)-n^j_k d_k^j+n^i_k d_k^j-\sum_{p\in D}\sum_{(\ell,s)\in I^{ij}_{k,p}}m^i_{k,p,\ell}m^j_{k,p,s}\right]-\\
\sum_{k=1}^h\left[n^j_k n^i_{k-1}(g-1)-n^j_k (d_{k-1}^i+|D|)+n^i_{k-1} d_k^j-\sum_{p\in D}\sum_{(\ell,s)\in \mathcal{I}^{ij}_{k,p}}m^{i}_{k-1,p,\ell}m^j_{k,p,s}\right],
\end{multline*}
where $I^{ij}_{k,p}=\{(\ell,s)\in[1,s^i_{k,p}]\times[1,s^j_{k,p}]|w_{k,p,\ell}^i<w_{k,p,s}^j\}$ and $\mathcal{I}^{ij}_{k,p}=\{(\ell,s)\in[1,s^i_{k-1,p}]\times[1,s^j_{k,p}]|w_{k-1,p,\ell}^i\leq w_{k,p,s}^j\}$.

Moreover in $\hat{K}_0(\mathrm{Var})$ we have 
$$[\underline{\e}(\underline{n}^h,\dots,\underline{n}^1)_{\underline{\mathcal{D}}^h,\dots,\underline{\mathcal{D}}^1}^{\underline{d}^h,\dots,\underline{d}^1,\mathrm{gr}\, \underline{\alpha}-ss}]
=\mathbb{L}^{\chi}\prod_{i=1}^h[\pc_{\underline{n}^i,\underline{\mathcal{D}}^i}^{\underline{d}^i,\underline{\alpha}-ss}].$$
\end{prop}
\begin{proof}
We shall prove this proposition by induction on $h$. If $h=1$ then there is nothing to prove.

Now, we consider $h\geq 2$ and the forgetful map 
{\small$$\mathrm{gr}_h:\underline{\e}(\underline{n}^{h-1},\dots,\underline{n}^1)_{\underline{\mathcal{D}}^{h},\dots,\underline{\mathcal{D}}^1}^{\underline{d}^{h},\dots,\underline{d}^1,\mathrm{gr}\, \underline{\alpha}-ss}\to \underline{\e}(\underline{n}^{h-1},\dots,\underline{n}^1)_{\underline{\mathcal{D}}^{h-1},\dots,\underline{\mathcal{D}}^1}^{\underline{d}^{h-1},\dots,\underline{d}^1,\mathrm{gr}\, \underline{\alpha}-ss}\times\pc_{\underline{n}^h,\underline{\mathcal{D}}^h}^{\underline{d}^h,\underline{\alpha}-ss}.$$}

We denote by $\mathrm{pr}_{ij}$ the projection from $$\underline{\e}(\underline{n}^{h-1},\dots,\underline{n}^1)_{\underline{\mathcal{D}}^{h-1},\dots,\underline{\mathcal{D}}^1}^{\underline{d}^{h-1},\dots,\underline{d}^1,\mathrm{gr}\, \underline{\alpha}-ss}\times \pc_{\underline{n}^h,\underline{\mathcal{D}}^h}^{\underline{d}^h,\underline{\alpha}-ss}\times C$$ onto 
the product of the $i$-th and $j$-th factor. Denote by $E'_{\bullet,\mathrm{univ}}, E^h_{\bullet,\mathrm{univ}} $ the universal parabolic chains on $\underline{\e}(\underline{n}^{h-1},\dots,\underline{n}^1)_{\underline{\mathcal{D}}^{h-1},\dots,\underline{\mathcal{D}}^1}^{\underline{d}^{h-1},\dots,\underline{d}^1,\mathrm{gr}\, \underline{\alpha}-ss}\times C$ and $\pc_{\underline{n}^h,\underline{\mathcal{D}}^h}^{\underline{d}^h,\underline{\alpha}-ss}\times C$.

Using the condition that $\pas_{\underline{\alpha}}(\underline{n}^i,\underline{d}^i,\underline{\mathcal{D}}^i)>\pas_{\underline{\alpha}}(\underline{n}^{i+1},\underline{d}^{i+1},\underline{\mathcal{D}}^{i+1})$ for $i=1,\dots,h-1$, we have $\pas_{\mathrm{min}}(E'_{\bullet,\mathrm{univ}})>\pas{\mathrm{max}}(E^h_{\bullet,\mathrm{univ}})$. By the first assertion of Lemma \ref{keylem}, we get that :
\begin{equation}\label{eqn3}
R\mathrm{pr}_{12,*}(C^\bullet(\mathrm{pr}_{23}^*E^h_{\bullet,\mathrm{univ}},\mathrm{pr}_{13}^*E'_{\bullet,\mathrm{univ}}))
\end{equation}

where $$C^\bullet(\mathbf{E''}_\bullet,\mathbf{E'}_\bullet):\oplus_{i}\ph(E''_i,E'_i)\buildrel c\over{\rightarrow} \oplus_{i}\sph(E''_i,E'_{i-1}(D));$$
is a complex with cohomology only in degree $0,1$. 

The proof of \cite[Corollary 3.2]{ghs} is still true for parabolic bundles with a small modification by taking $\ph$ or $\sph$ instead of $\h$ and the closed points $p\in C-D$. The dimension of forget full map now is the Euler characteristic of the sheaf $\ph$ or of the sheaf $\sph$ . Using this point of view we see that the complex (\ref{eqn3}) can be represented by a complex of vector bundles $\mathcal{F}_0\buildrel d_0\over{\rightarrow}\mathcal{F}_1\buildrel d_1\over{\rightarrow}\mathcal{F}_2$. Its cohomology is only in degree 0,1 implies that it is quasi-isomorphic to $\mathcal{F}_0\to \ker(d_1)$. Using the \cite[Proposition 3.1]{ghs}, we have that the vector bundle stack $[\ker(d_1)/\mathcal{F}_0]$ is isomorphic to
$\underline{\e}(\underline{n}^{h-1},\dots,\underline{n}^1)_{\underline{\mathcal{D}}^{h-1},\dots,\underline{\mathcal{D}}^1}^{\underline{d}^{h-1},\dots,\underline{d}^1,\mathrm{gr}\, \underline{\alpha}-ss}$. 

Let $E$, $F$ two parabolic bundles of weight datum $\mathcal{D}^E=(\underline{s}^E,\underline{w}^E,\underline{m}^E)$ and $\mathcal{D}^F=(\underline{s}^F,\underline{w}^F,\underline{m}^F)$. We have in mind that there are natural skyscraper sheaves $\oplus_{p\in D}\mathcal{K}_p$ and $\oplus_{p\in D}\mathcal{SK}_p$ supported on the points of $D$ such that  
$$0\to \ph(E,F)\to \ho(E,F)\to \bigoplus_{p\in D}\mathcal{K}_p\to 0,$$
and $$0\to \sph(E,F)\to \ho(E,F)\to \bigoplus_{p\in D}\mathcal{SK}_p\to 0$$
are a short exact sequence of sheaves. Using the proof of \cite[Lemma 2.4]{bh}, we obtain the formula for the Euler characteristic of $\mathcal{K}_p$ and $\mathcal{SK}_p$. 
More precisely, we have 
$$\chi({\mathcal{K}_p})=\sum_{(i,j)\in \{1\leq i\leq s^E_p,1\leq j\leq s^F_p|w_{p,i}^E>w_{p,j}^F\}}m^E_{p,i}m^F_{p,j},$$
and 
$$\chi({\mathcal{SK}_p})=\sum_{(i,j)\in \{1\leq i\leq s^E_p,1\leq j\leq s^F_p|w_{p,i}^E\geq w_{p,j}^F\}}m^E_{p,i}m^F_{p,j}.$$

Using the Riemann-Roch formula, we have $$\chi(\ho(E,F))=\rk(E)\rk(F)(g-1)+\rk(E)\deg(F)-\rk(F)\deg(E).$$
Thus we can calculate the Euler characteristic of $\ph(E,F)$ and $\sph(E,F)$ as follows :
$$\chi(\ph(E,F))=\chi(\ho(E,F))-\sum_{p\in D}\chi(\mathcal{K}_p),$$ 
$$\chi(\sph(E,F))=\chi(\ho(E,F)-\sum_{p\in D}\chi({\mathcal{SK}_p}).$$

Applying these formulas into our situation, we obtain the formula for the dimension of the fibres of $\mathrm{gr}_{h}$ is $\sum_{j=1}^{h-1}\chi_{jh}$. 

Using the hypothesis of induction, the result is claimed.
  
\end{proof}
\subsection{Parabolic chains of rank $(n,\dots,n)$}
In this section we give an inductive formula for the stacks of parabolic chains $\underline{\alpha}$- semistable of rank $(n,\dots,n)$, degree $\underline{d}$ and of weight data $\underline{\mathcal{D}}$ in the special case when $d_{i-1}-d_i+2n|D|<\alpha_i-\alpha_{i-1}$.

Before going further, we need to recall some results of \cite[Section 3]{ghs} which we shall use in our calculation.

For any family of vector bundle $E$ of rank $n$ parameterized by a scheme of finite type (or stack of finite type with affine stabilizer groups) $T$, we have:
\begin{prop}(cf. \cite[Proposition 3.6]{ghs})\label{flag2}
The class of the stack $\mathrm{Hecke}(E/T)^\ell$ classifying modification $E'\subset E$ with $E/E'$ a torsion sheaf of length $\ell$ is:
$$[\mathrm{Hecke}(E/T)^\ell]=[T]\times [(C\times \mathbb{P}^{n-1})^{(\ell)}].$$ 
\end{prop}

Now we shall adapt this proposition in our situation. For any of parabolic vector bundle of rank $n$ and of weight data $\mathcal{D}$ parameterized by $T$ (like as above), we shall write $\mathrm{PHecke}(E/T)^\ell_{\mathcal{D}'}$ for the stack classifying modifications $E'\subset E$ with $E'$ a parabolic sub-bundle of weight data $\mathcal{D}'$ and $E/E'$ a torsion sheaf of length $\ell$. 

\begin{cor}\label{corhecke}
We have: 
\begin{eqnarray*}
[\mathrm{PHecke}(E/T)^\ell_{\mathcal{D}'}]&=&[\mathrm{Hecke}(E/T)^\ell]\times\prod_{p\in D}[\mathrm{Flag}(n,s_p,m'_{p,i})]\\
&=&[T]\times [(C\times \mathbb{P}^{n-1})^{(\ell)}] \times \\
&&\prod_{p\in D}\frac{\prod_{i=0}^{n-1}(\mathbb{L}^n-\mathbb{L}^i)}{\left[\prod_{i=1}^{s_p}\prod_{j=0}^{m'_{p,i}-1}(\mathbb{L}^{m'_{p,i}}-\mathbb{L}^j)\right]\left[\mathbb{L}^{\sum_{1\leq i<j\leq s_p}m'_{p,i}m'_{p,j}}\right]}.
\end{eqnarray*}
\end{cor}

\begin{proof}
We consider the forgetful map $\mathrm{PHecke}(E/T)^\ell_{\mathcal{D'}}\to \mathrm{Hecke}(E/T)^\ell$. The fibres of this map are products of flag manifolds $\prod_{p\in D}\mathrm{Flag}(n,s_p,m'_{p,i})$. The claim is a consequent of the Proposition \ref{flag2} and the lemma \ref{flag1}. 
\end{proof}

Given $\underline{n},\,\underline{d},\,\underline{\mathcal{D}}$ and $\underline{l}$, let us denote by $\pc_{\underline{n},\underline{\mathcal{D}}}^{\underline{d},\mathrm{inj}}$ the stack of parabolic chains of rank $\underline{n}$, degree $\underline{d}$ and weight data $\underline{\mathcal{D}}$, such that all maps $\phi_i$ are injective.

\begin{prop}\label{p10}
Fix $n,r\in \mathbb{N}$ and write $\underline{n}=(n,\dots,n)$. Fix a degree $\underline{d}$ 
and $\underline{\alpha}$ a semi-stability parameter. Suppose that for all $i>0$ we have $d_{i-1}-d_i+2n|D|<\alpha_i-\alpha_{i-1}$, then :
\begin{enumerate}
\item For any $\underline{\alpha}$-semi-stable parabolic chain of rank $\underline{n}$ and degree $\underline{d}$ all maps $\phi_i$ are injective, i.e. , $\pc_{\underline{n},\underline{\mathcal{D}}}^{\underline{d},\underline{\alpha}-ss}\subset \pc_{\underline{n},\underline{\mathcal{D}}}^{\underline{d},\mathrm{inj}}$.
\item Suppose $\mathbf{E^r_\bullet}\in \pc_{\underline{n},\underline{\mathcal{D}}}^{\underline{d},\mathrm{inj}}$ is a parabolic chain with HN-filtration $\mathbf{E^r_\bullet}^{(1)}\subset \dots \subset \mathbf{E^r_\bullet}^{(h)}=\mathbf{E^r_\bullet}$. Then for any $i$ we have $\rk(\mathbf{E^r_\bullet}^{(i)}/\mathbf{E^r_\bullet}^{(i-1)})=(m_i,\dots, m_i)$ for some $m_i\in \mathbb{N}$.
\end{enumerate} 
\end{prop}

\begin{proof}
To show (1) suppose $\mathbf{E^r_\bullet}$ was a $\underline{\alpha}$-semi-stable with $n>\rk(\ker(\phi_i))=m>0$ for some $i$. Then $$(0,\dots,0,\ker(\phi_i),0,\dots,0)$$ and $$(E_r,\dots,E_i,(E_i/\ker(\phi_i))(-D),E_{i-2},\dots,E_0)$$ are parabolic sub-chains of $\mathbf{E^r_\bullet}$. Thus by using the condition of semi-stability we have :
$$(r+1)n.\pd(\ker(\phi_i))\leq m\sum_{j=0}^r(\pd(E_j)+n\alpha_j)-m(r+1)n\alpha_i$$
and 
\begin{multline*}
m\sum_{j=0}^r(\pd(E_j)+n\alpha_j)+ (r+1)n(\pd(E_i)-\pd(E_{i-1})-m\alpha_{i-1}-(n-m)|D|)\\
\leq (r+1)n.\pd(\ker(\phi_i)).
\end{multline*}
This implies 
$$\alpha_i-\alpha_{i-1}\leq \frac{\pd(E_{i-1})-\pd(E_{i})}{m}+\frac{n-m}{m}|D|.$$

Remark that by definition, the weights of parabolic bundles are belong to $(0,1)$ and the sum of the multiplicity at a point $p$ is equal to the rank of bundle, so we have $d_i<\pd(E_i)<d_i+n|D|\,\forall i$. Using this in the inequality above, we get a contradiction. 

We shall prove (2) by induction on $h$. Suppose $\mathbf{E'^{r}_\bullet}=\mathbf{E^{r}_\bullet}^{(i)}\subset \mathbf{E^r_\bullet}$ was a destabilizing parabolic sub-chain such that not all $E^{(i)}_j $ have equal rank. We shall denote by $\mathbf{E''^{r}_\bullet}=\mathbf{E^{r}_\bullet}/\mathbf{E'^{r}_\bullet}$ the quotient parabolic chain.

By assumption all maps $\phi_j^{(i)}$ are injective, so that $n'_{j}\leq n'_{j-1}\leq n\,\forall j\in\{1,\dots,r\}$. Let $j$ be the minimal integer such that $ n'_{j}< n'_{j-1}$.

Then $\mathbf{K^r_\bullet}:=(0,\dots,O,\ker(\phi''_j),0,\dots,0)$ is a parabolic sub-chain of $\mathbf{E''^{r}_\bullet}$ and $\mathbf{Q_\bullet^r}:=(0,\dots,0,E'_{j-1}/E'_j(-D),0,\dots,0)$ is a quotient of $\mathbf{E'^{r}_\bullet}$. Thus we have:
$$\pas(\mathbf{Q^{r}_\bullet})\geq\pas_{\min}(\mathbf{E'^{r}_\bullet})>\pas_{\max}(\mathbf{E''^{r}_\bullet})\geq \pas(\mathbf{K_\bullet}),$$
i.e,
\begin{eqnarray*}
&&\frac{\pd(E'_{j-1})-\pd(E'_j)+n'_j|D|}{n'_{j-1}-n_j}+\alpha_{j-1}>\pas(\mathbf{K_\bullet})\geq \\
&&\quad\quad\quad \quad \quad\quad\quad\quad\quad\quad\frac{\pd(E''_j)-\pd(E''_{j-1})-|D|n''_{j-1}}{n''_j-n''_{j-1}}+\alpha_j\\
&\Rightarrow &\frac{\pd(E'_{j-1})-\pd(E'_j)+n'_j|D|}{n'_{j-1}-n_j}+\alpha_{j-1}>\\
&&\frac{\pd(E_j)-\pd(E_{j-1})+\pd(E_{j-1})-\pd(E_{j})-|D|(n-n'_{j-1})}{n'_{j-1}-n'_{j}}+\alpha_j\\
&\Rightarrow &\frac{\pd(E_{j-1})-\pd(E_j)-|D|(n-n'_{j-1})}{n'_{j-1}-n'_{j}}>\alpha_j-\alpha_{j-1}.
\end{eqnarray*}

By using the remark above, we get again a contradiction.  
\end{proof}

This proposition allows us to have the following recursion formula for the motive of $\pc_{\underline n,\mathcal{D}}^{\underline{d},\underline{\alpha}-ss}$ in the special cases where $d_{i-1}-d_i+2n|D|<\alpha_{i}-\alpha_{i-1}$ and $\underline{n}=(n,\dots,n)$.
\begin{cor}\label{cors}
Let $\underline{n}=(n,\dots,n)$ be constant. If $\underline{\alpha},\underline{d}$ satisfy $n|D|\leq d_{i-1}-d_i+2n|D|<\alpha_{i}-\alpha_{i-1}$ for all $i\geq 1$ then we have 
\begin{multline*}
[\pc_{\underline n,\underline{\mathcal{D}}}^{\underline{d},\underline{\alpha}-ss}]=[\mathrm{PBundle}_{n}^{d_0}]\prod_{i=1}^r\left([(C\times \mathbb{P}^{n-1})^{(d_{i-1}-d_i+n|D|)}]\prod_{p\in D}[\mathrm{Flag}(n,s_{i,p},m_{i,p,j}]\right)\\
-\left(\sum_{\underline{m},\underline{e},\underline{k},\underline{\mathcal{D}}^\bullet}\mathbb{L}^{\sum_{k<j}\chi_{kj}}\prod_{j}[\pc_{\underline{m}_j,\underline{\mathcal{D}}^{(j)}}^{\underline{e}^{(j)},\underline{\alpha}-ss}]\right) 
\end{multline*}
where the sum runs over all partitions $n=\sum_{j=1}^h m_j, \underline{d}=\sum_j\underline{e}^{(j)}$ and $\underline{\mathcal{D}}=\sum_{j}\underline{\mathcal{D}}^{(j)}$ (here we use the trick of Holla (cf. \ref{rmkw} to see that the HN filtration inherit the same collection of weight and ``+" means the sum of multiplicity) such that $\frac{\sum_i\pd_i^{(j)}}{rm_j}>\frac{\sum_i\pd_i^{(j+1)}}{rm_{j+1}}$ ($\pd_i^{(j)}$ is parabolic degree defined with the help of $\underline{e}^{(j)}$ and $\underline{\mathcal{D}}^{(j)}$) and $\sum_{\ell=1}^{s_{i,p}^{(j)}}m_{i,p,\ell}^{(j)}=m_j\,\forall i\in\{1,\dots,r\}$. We have written 
\begin{multline*}\chi_{kj}=m_jm_k(g-1)+\sum_{i=0}^r\left(m_ke_i^{(j)}-m_je_i^{(k)}\sum_{p\in D}\sum_{(\ell,s)\in I^{jk}_{i,p}}m^{(j)}_{i,p,\ell}m^{(k)}_{i,p,s}\right)\\-\sum_{i=1}^r\left(m_ke_i^{(j)}-m_j(e_{i-1}^{(k)}+|D|)-\sum_{p\in D}\sum_{(\ell,s)\in \mathcal{I}^{jk}_{i,p}}m^{(j)}_{i,p,\ell}m^{(k)}_{i-1,p,s}\right),
\end{multline*}
where $I^{kj}_{i,p}=\{(\ell,s)\in[1,s^{(k)}_{i,p}]\times[1,s^{(j)}_{i,p}]|w_{i,p,\ell}^{(k)}<w_{i,p,s}^{(j)}\}$ and $\mathcal{I}^{kj}_{i,p}=\{(\ell,s)\in[1,s^{(k)}_{i-1,p}]\times[1,s^{(j)}_{i,p}]|w_{i-1,p,\ell}^{(k)}\leq w_{i,p,s}^{(j)}\}$.
\end{cor}
\begin{proof}
The Proposition \ref{p10} shows us that all $\underline{\alpha}$-semi-stable parabolic chains is contained in the sub-stack of parabolic chains such that all maps $\phi_i$ are injective and moreover for any such parabolic chains all sub-quotients of HN-filtration also satisfy this condition.

Thus we have 
$$[\pc_{\underline n,\underline{\mathcal{D}}}^{\underline{d},\underline{\alpha}-ss}]=[\pc_{\underline n,\underline{\mathcal{D}}}^{\underline{d},\mathrm{inj}}]-\bigcup \text{Harder-Narasimhan strata}.$$

The stack $[\pc_{\underline n,\underline{\mathcal{D}}}^{\underline{d},\mathrm{inj}}]$ classifies a collection of Hecke modification $E_i\subset E_{i-1}(D)$ of length $d_{i-1}+n|D|-d_i$. Its classe is computed by using the formula of the Corollary \ref{corhecke}.

The Harder-Narasimhan strata are given by ranks, degrees and weight datum as claimed. Since in all occurring sub-quotients the maps $\phi$ are injective, we can apply the Proposition \ref{p6} and the argument of Proposition \ref{p7} for computing their class.    
\end{proof}

\section{Crossing the wall of critical values}
\subsection{Critical values}
A parabolic chain $\bf{E^r_\bullet}$ of fixed topological and parabolic type is strictly $\underline{\alpha}$-semi-stable if and only if it has a proper parabolic sub-chain $\bf{E'^r_\bullet}$ such that $\pas_{\underline{\alpha}}(\bf{E'^r_\bullet})=\pas_{\underline{\alpha}}(\bf{E^r_\bullet})$, i.e
\begin{equation}\label{eqn1}\sum_{i=0}^r\frac{\rk(E'_i)}{|\rk(\mathbf{{E'}_\bullet^r})|}(\pas(E'_i)+\alpha_i)=\sum_{i=0}^r\frac{\rk(E_i)}{|\rk(\mathbf{{E}_\bullet^r})|}(\pas(E_i)+\alpha_i).
\end{equation}
There are two ways in which this can happen. The first one is if there exists a parabolic sub-chain $\bf{E'^r_\bullet}$ such that
$$\frac{\rk(E'_i)}{|\rk(\mathbf{{E'}_\bullet^r})|}=\frac{\rk(E_i)}{|\rk(\mathbf{{E}_\bullet^r})|},\,\forall i\in\{0,\dots,r\}\text{ and}$$
$$\pas(\oplus_{i=0}^rE'_i)=\pas(\oplus_{i=0}^rE_i).$$
In this case $\bf{E^r_\bullet}$ is strictly $\underline{\alpha}$-semi-stable for all values of $\underline{\alpha}$. This phenomemon are so-called \textit{$\underline{\alpha}$-independent semi-stability}.

The other way in which strict $\underline{\alpha}$-semi-stability can happen is if the equation \ref{eqn1} holds but
$$\frac{\rk(E'_i)}{|\rk(\mathbf{{E'}_\bullet^r})|}\neq\frac{\rk(E_i)}{|\rk(\mathbf{{E}_\bullet^r})|}$$
for some $i\in\{0,\dots,r\}$. The values of $\underline{\alpha}$ for which this happens are called \textit{critical values}.

From now on we shall make the following assumption on the weights.
\begin{asum}\label{asum1}Let $\{w_{p,i}\}$ be the collection of all the weights of $E_i$ together. We assume that they are all of multiplicity one and that, for a large integer $N$ depending only on the ranks, they following property :
$$\sum_{1\leq i\leq r,p\in D}n_{p,i}w_{p,i}\in \mathbb{Z},\quad n_{p,i}\in \mathbb{Z},|n_{p,i}|\leq N \Rightarrow n_{p,i}=0,\,\forall p,i. $$
The weights failing this condition are a finite union of hyperplane in \linebreak $[0,1)^{|D|.\sum_{i=0}^r\rk(E_i)}$.

The assumption that the weights are all of multiplicity only helps us to state the property more easily. In fact, we can remove this condition by repeating each weight according to its multiplicity and seeing them as the different weights.
\end{asum}
\begin{lem}\label{lem2}
Under Assumption \ref{asum1}, there are no $\underline{\alpha}$-independent semi-stable parabolic chains of rank $\underline{n}$ (by taking $N$ larger than $\sum_{i=0}^rn_i$).
\end{lem}
\begin{proof}
Using the assumption \ref{asum1}, we have $\pas(\oplus E'_i)\neq \pas(\oplus E_i)$ for all parabolic sub-chains $\bf{E'^r_\bullet}$ of $\bf{E^r_\bullet}$. So the lemma is its consequence.
\end{proof}
If the rank, the degree and the weight data are fixed, equation \ref{eqn1} requires that the point $\underline{\alpha}$ belong to the intersection of an affine hyperplane with the space of semi-stability parameters. 
Call this intersection a \textit{wall}. 
In general the number of wall is infinite event in the special case - parabolic triples (see. \cite[Proposition 5.2 (ii)]{ggm}).

When varying the stability parameter such that the parameter crosses a wall, the change of the moduli spaces can be described in principle. Such description can be difficult to obtain on the level of coarse moduli space cause of the existence of poly-stable objects, but easier on the level of moduli stacks. So, we shall study the behavior of the moduli stack when semi-stability parameter runs into a wall and out of it.
\subsection{Wall crossing from the point of view of stacks}
For a parabolic chains of fixed weight datum, we can alway introduce a new parabolic structure for each parabolic bundle which have the same collection of weights- this collection contain all the weights of each parabolic bundle- and keep track all the information of old parabolic structures, by using the point of view of \cite[Section 2.2]{by}. By using the trick of Holla (see remark \ref{rmkw}), we can say also about the collection of weights of its Harder-Narasimhan filstration. Using this point of view, we consider the following proposition.

\begin{prop}\label{pkey}
Let $\underline{\alpha}$ be a critical value, let $\underline{\delta}\in \mathbb{R}^{r+1}$, fix a rank $\underline{n}\in \mathbb{N}^{r+1}$ and fix weights $\{w_{p,i}\}$.
\begin{enumerate}
\item There exists $\epsilon>0$ such that for all $0<t<\epsilon$, the semi-stability conditions $\underline{\alpha}_t=\underline{\alpha}+t\underline{\delta}$ coincide for all parabolic chains of rank $\underline{m}$ if $\underline{m}\leq \underline{n}$.
\item For any $0<t<\epsilon$ as in (1) we have 
$$\pc_{\underline{n},\underline{\mathcal{D}}}^{\underline{d},\underline{\alpha}_t-ss}\subset \pc_{\underline{n},\underline{\mathcal{D}}}^{\underline{d},\underline{\alpha}-ss}.$$
Moreover, the complement $\pc_{\underline{n},\underline{\mathcal{D}}}^{\underline{d},\underline{\alpha}-ss}-\pc_{\underline{n},\underline{\mathcal{D}}}^{\underline{d},\underline{\alpha}_t-ss}$ is the finite union of the $\underline{\alpha}_t$-Harder-Narasimhan strata of $\pc_{\underline{n},\underline{\mathcal{D}}}^{\underline{d}}$ of type $(\underline{n}^i,\underline{d}^i,\underline{\mathcal{D}}^i)$ such that 
$\pas_{\underline{\alpha}}(\underline{n}^i,\underline{d}^i,\underline{\mathcal{D}}^i)=\pas_{\underline{\alpha}}(\underline{n},\underline{d},\underline{\mathcal{D}})$ for all $i$.
\end{enumerate}
\end{prop}
\begin{proof}
The proof is similar the one in \cite[Proposition 2]{gh}. The proof of the first claim require some care on the weights so we shall write it down.

To simplify notation, let us abbreviate the $\underline{\alpha}_t$-slope by $\pas_t:=\pas_{\underline{\alpha}_t}$. We shall prove by using disproof. If this $\epsilon$ doesn't exist, means that for every $\epsilon>0$, there exist $0<x_1<x_2<\epsilon$ such that $\underline{\alpha}_{x_1}$ and $\underline{\alpha}_{x_2}$ define different semi-stability conditions.
So there exist parabolic chains $\bf{F^r_\bullet}\subset\bf{E^r_\bullet}$ such that either $\pas_{x_1}(\mathbf{F^r_\bullet})>\pas_{x_1}(\mathbf{E^r_\bullet})$ and $\pas_{x_2}(\mathbf{F^r_\bullet})\leq\pas_{x_2}(\mathbf{E^r_\bullet})$ or $\pas_{x_1}(\mathbf{F^r_\bullet})\leq\pas_{x_1}(\mathbf{E^r_\bullet})$ and $\pas_{x_2}(\mathbf{F^r_\bullet})>\pas_{x_2}(\mathbf{E^r_\bullet})$.
However, $\pas_{t}(\mathbf{F^r_\bullet})-\pas_{t}(\mathbf{E^r_\bullet})$ is continuous over $t\in \mathbb{R}$, so there exist $x_1\leq x\leq x_2$ such that $\pas_{t}(\mathbf{F^r_\bullet})-\pas_{t}(\mathbf{E^r_\bullet})=0$.

By definition $\pas_{t}(\mathbf{F^r_\bullet})-\pas_{t}(\mathbf{E^r_\bullet})=c+f(\{w_{p,i}\})+x.\frac{m_i\delta_i}{M}$, where  $c$ is a rational number with denominator bounded by $|\underline{n}|(|\underline{n}|-1)$, 
$f$ is a linear combinaison of $w_{p,i}$ with rational coefficients who are bounded by  $2|\underline{n}|$, and where $m_i,M$ are integers satisfying $M\leq |\underline{n}|(|\underline{n}|-1)$ and $m_i\leq |\underline{n}|(|\underline{n}|-1)$. Therefore, $x$ must lie in a discrete subset of $\mathbb{R}$ which is conflict with the condition that we can choose $\epsilon>0$ as small as possible. Consequently, the first assertion is claimed.

\end{proof}
\section{Moduli stacks of parabolic chains : recursion formulas}
In this section we shall explain our strategy to compute the motives of moduli spaces of parabolic chains $\underline{\alpha}$-semi-stable which satisfy the condition $\alpha_i-\alpha_{i-1}\geq 2g-2$.

The first step is applying our necessary conditions to find - for any rank $\underline{n}$, degree $\underline{d}$ and weight datum $\underline{\mathcal{D}}$ and any parameter $\underline{\alpha}$ satisfying $\alpha_i-\alpha_{i-1}\geq 2g-2$ - a family of stability conditions $\underline{\alpha}_t=\underline{\alpha}+t\underline{\delta}$ such that for all $t\geq 0$, the parameter $\underline{\alpha}_t$ satisfies $\alpha_{t,i}-\alpha_{t,i-1}\geq 2g-2$, and such that we can compute $\pc_{\underline{n},\underline{\mathcal{D}}}^{\underline{d},\underline{\alpha}_t-ss}$ for $t$ large enough.

Let us first assume that $n_i\neq n_j$ for some $i,j$. Then we can find a such family $\underline{\alpha}_t$ such that the moduli spaces is empty for $t$ large enough.
\begin{lem}
Fix $\underline{n}\in \mathbb{N}^{r+1},\,\underline{d}\in \mathbb{Z}^{r+1}, \underline{\mathcal{D}}$ together with a semi-stability $\underline{\alpha}\in\mathbb{R}^{r+1}$ satisfying $\alpha_i-\alpha_{i-1}\geq 2g-2$. Suppose that $n_r\neq n_i$ for some $0\leq i<r$. Then we have :
\begin{enumerate}
\item If $n_r=n_{r-1}=\dots=n_{k+1}<n_k$ for some $k$, define $\delta_i:=\begin{cases}
1 \text{ if } i>k\\
0 \text{ otherwise.} \end{cases}$
\item If $n_r=n_{r-1}=\dots=n_{k+1}>n_k$ for some $k$, define $\delta_i:=\begin{cases}
0 \text{ if } i>k\\
-1 \text{ otherwise.} \end{cases}$.
\end{enumerate}
Let $\underline{\alpha}_t:=\underline{\alpha}+t\underline{\delta}$. Then $\alpha_{t,i}-\alpha_{t,i-1}\geq 2g-2$ for all $t\geq 0$ and $i\in \{1,\dots,r\}$ and $\pc_{\underline{n},\underline{\mathcal{D}}}^{\underline{d},\underline{\alpha}_t-ss}=0$ for $t$ large enough.
\end{lem}
\begin{proof}
In both case, we have $\delta_i\geq \delta_{i-1}$. So 
$$\alpha_{t,i}-\alpha_{t,i-1}=(\alpha_i-\alpha_{i-1})+t(\delta_i-\delta_{i-1})\geq (\alpha_i-\alpha_{i-1})\geq 2g-2.$$ 

In case (1), we have (because $\delta_i\geq 0$ and $\underline{\delta}\neq 0$)
$$\lim_{t\to \infty}\pas_{\underline{\alpha}_t}(\underline{n},\underline{d},\underline{\mathcal{D}})=\lim_{t\to \infty}\frac{\sum_{i=0}^r(\pd_i+(\alpha_i+t\delta_i)n_i)}{|n|}=+\infty.$$
For $j=k+1$ the condition (3) of the Proposition \ref{p4} is equivalent to 
$$\pas_{\underline{\alpha}_t}(\underline{n},\underline{d},\underline{\mathcal{D}})\leq \frac{\pd_{k}-\pd_{k+1}+n_k|D|}{n_{k}-n_{k+1}}+\alpha_k+t\delta_k.$$
Since $\delta_k=0$, the right hand side is independent of $t$, so that this condition cannot hold for $t$ large enough.

In case (2), we have 
$$\lim_{t\to \infty}\pas_{\underline{\alpha}_t}(\underline{n},\underline{d},\underline{\mathcal{D}})=\lim_{t\to \infty}\frac{\sum_{i=0}^r(\pd_i+(\alpha_i+t\delta_i)n_i)}{|n|}=-\infty.$$
Using the condition (4) of the Proposition \ref{p4} for $j=k+1$ we have 
$$\frac{\pd_{k+1}-\pd_{k}-n_k|D|}{n_{k+1}-n_{k}}+\alpha_{k+1}+t\delta_{k+1}\leq \pas_{\underline{\alpha}_t}(\underline{n},\underline{d},\underline{\mathcal{D}}).$$
Since $\delta_{k+1}=0$, the left hand side is independent of $t$, so that this condition cannot hold for $t$ large enough.  
\end{proof}

If $\underline{n}=(n,\dots,n)$, we have the following lemma 
\begin{lem}
Let $\underline{n}=(n,\dots,n)$ and let $\underline{\alpha}$ satisfying $\alpha_i-\alpha_{i-1}\geq 2g-2$. Let $\delta_i=i$, and put $\underline{\alpha}_t:=\underline{\alpha}+t\underline{\delta}$. Then $\alpha_{t,i}-\alpha_{t,i-1}\geq 2g-2$ for all $t\geq 0$, and for $t$ large enough the parameter $\underline{\alpha}_t$ satisfies $\alpha_{t,i}-\alpha_{t,i-1}>d_{i-1}-d_i+2n|D|$. For such $t$ we have that $[\pc_{\underline n,\mathcal{D}}^{\underline{d},\underline{\alpha}_t-ss}]$ is calculated from the Corollary \ref{cors}. 
\end{lem}

The second step, we can now apply the wall-crossing procedure of above section. The Proposition \ref{pkey} implies that the difference  $[\pc_{\underline n,\mathcal{D}}^{\underline{d},\underline{\alpha}_t-ss}]-[\pc_{\underline n,\mathcal{D}}^{\underline{d},\underline{\alpha}-ss}]$ is given by an alternating sum of classes of finitely many Harder-Narasimhan strata for some stability parameters $\underline{\alpha}_s$ with $0\leq s\leq t$. Since $\underline{\alpha}_s$ satisfies $\alpha_{s,i}-\alpha_{s,i-1}\geq 2g-2$ for all $i$, so the classes of the Harder-Narasimhan strata can be computed by the Proposition \ref{p7}

By induction on the rank $\underline{n}$, this process give us a recursive formula for the class of $\pc_{\underline n,\mathcal{D}}^{\underline{d},\underline{\alpha}-ss}$. 

By using the Corollary \ref{corm} and the Corollary \ref{cor1} for $\underline{\alpha}=(0,2g-2,\dots, r(2g-2))$, we can determine the class $[M^d_{n,\mathcal{D}}]$ with finite calculation.

\begin{cor}
The class $[M^d_{n,\mathcal{D}}]$ can be expressed in terms of $\mathbb{L}$ and the symmetric power $C^{(i)}$ of the curve $C$. 
\end{cor}

\end{document}